\documentclass[12pt]{amsart}
\usepackage{amscd, amsfonts, amssymb, amsthm, mathrsfs, amsmath, mathtools, flexisym}
\usepackage{parskip, fullpage, verbatim}
\usepackage[all]{xypic}
\usepackage[colorlinks, breaklinks, linkcolor=blue]{hyperref} 
\usepackage{breakurl}
\usepackage{ctable}
\usepackage{longtable,aliascnt}
\usepackage{array}
\usepackage{booktabs}
\usepackage{wasysym} 
\usepackage{float}
\restylefloat{table}
\usepackage{setspace}
\usepackage{multirow}
\usepackage{romannum}
\usepackage{csquotes}
\MakeOuterQuote{"}

\newcommand\CA{{\mathscr A}} 
\newcommand\CB{{\mathscr B}}

\newcommand\CO{{\mathcal O}}

\newcommand\BBC{{\mathbb C}}

\newcommand\BBP{{\mathbb P}}

\newcommand\GL{\operatorname{GL}}

\renewcommand\th{{^{\text{th}}}}

\numberwithin{equation}{section}

\theoremstyle{plain}
\newtheorem{theorem}{Theorem}
\newtheorem{lemma}[equation]{Lemma}

\newtheorem*{Hope}{Hope}

\theoremstyle{definition}

\subjclass[2010]{Primary 20F55, 52C35, 14N20; Secondary 13N15}

\begin{document}

\title{On the $K(\pi, 1)$-problem for restrictions of complex reflection arrangements}

\author[N. Amend]{Nils Amend}
\address
{Institut f\"ur Algebra,~Zahlentheorie und Diskrete Mathematik,
Fakult\"at f\"ur Mathematik und Physik,
Gottfried Wilhelm Leibniz Universit\"at Hannover,
Welfengarten 1, D-30167 Hannover, Germany}
\email{amend@math.uni-hannover.de}

\author[P. Deligne]{Pierre Deligne}
\address
{Institute for Advanced Study, 
1 Einstein Drive,
Princeton, New Jersey
08540, USA}
\email{deligne@math.ias.edu}

\author[G. R\"ohrle]{Gerhard R\"ohrle}
\address
{Fakult\"at f\"ur Mathematik,
Ruhr-Universit\"at Bochum,
Universit\"atsstra{\ss}e 150,
D-44780 Bochum, Germany}
\email{gerhard.roehrle@rub.de}

\keywords{Complex reflection groups,
	reflection arrangements, 
	restricted arrangements, 
	$K(\pi,1)$-arrangements,
	Eilenberg--MacLane space}

\pagenumbering{arabic}
\allowdisplaybreaks

\begin{abstract}
Let $W\subset \GL(V)$ be a complex reflection group, and $\CA(W)$ the set of the mirrors 
of the complex reflections in $W$. 
It is known that the complement $X(\CA(W))$ of the reflection arrangement $\CA(W)$ is a $K(\pi,1)$ space.

For $Y$ an intersection of hyperplanes in $\mathscr A(W)$, let $X(\mathscr A(W)^Y)$ be the complement in $Y$ of the hyperplanes in $\mathscr A(W)$ not containing $Y$. 
We hope that $X(\mathscr A(W)^Y)$ is always a $K(\pi,1)$. We prove it in case 
of the monomial groups $W = G(r,p,\ell)$.
Using known results, we then show that there remain only three irreducible complex 
reflection groups, leading to just eight such induced arrangements for which this 
$K(\pi,1)$ property remains to be proved.
\end{abstract}

\maketitle

\section{Introduction}
An \emph{arrangement} in a vector space $V$ is a finite set of homogeneous hyperplanes in $V$. For integers $\ell \geqslant 2,\ 0 \leq k \leq \ell$ and $r\geqslant1$, we define $\mathscr A^k_{\ell}(r)$ to be the arrangement in the complex vector space $\mathbb{C}^{\ell}$ (coordinates $y_1, \ldots,y_{\ell}$) consisting of the first $k$ coordinate hyperplanes $y_a=0 \ (1\leq a\leq k)$ and of the hyperplanes $y_i=\zeta y_j$ for $i\ne j$ and $\zeta$ an $r\th$ root of unity.

For $\mathscr A$ an arrangement in a complex vector space $V$, we define $X(\mathscr A)$ to be the complement in $V$ of the union of the hyperplanes in $\mathscr A$. We say that the arrangement $\mathscr A$ is of $K(\pi,1)$ type, or a $K(\pi,1)$-arrangement, if $X(\mathscr A)$ is a $K(\pi,1)$, that is if the homotopy groups $\pi_i$ of $X(\mathscr A)$ are trivial for $i\geqslant 2$, or equivalently if the universal covering of $X(\mathscr A)$ is contractible. Our main result is the following.

\begin{theorem}
\label{thm:main}
The arrangements $\CA^k_\ell(r)$ are of $K(\pi,1)$ type.
\end{theorem}

For $V$ a finite dimensional complex vector space, an element $s$ of $\GL(V)$ is a \emph{complex reflection} if its fixed point set is a hyperplane. This hyperplane is called the \emph{mirror} of $s$. If $W \subset \GL(V)$ is a complex reflection group, that is a finite subgroup of $\GL(V)$ generated by complex reflections, the \emph{arrangement} $\mathscr A(W)$ is the set of mirrors of the complex reflections in $W$. The arrangements so obtained are the \emph{reflection arrangements}.  
Note that for $\ell \ge 3$, $r \ge 2$, and  $0 < k < \ell$,
the $\mathscr A^k_{\ell}(r)$ are not reflection arrangements.

If $X\subset V$ is the intersection of some hyperplanes belonging to an arrangement $\CA$ in $V$, the arrangement $\CA^X$ \emph{induced} by $\CA$ on $X$ is the set of the traces on $X$ of the hyperplanes in $\CA$ not containing $X$. The arrangements \emph{induced} by $\CA$ are the arrangements so obtained.

The $K(\pi, 1)$ property is not generic among all arrangements.
A \emph{generic} complex 
$\ell$-arrange\-ment $\CA$ for $\ell \ge 3$
is an $\ell$-arrangement with at least $\ell + 1$ hyperplanes and
the property that the hyperplanes of every subarrangement
$\CB \subseteq \CA$ with $\vert\CB\vert = \ell$
are linearly independent. It follows from work of Hattori \cite{hattori} that
generic arrangements are never $K(\pi, 1)$.

By Deligne \cite{deligne}, complexified simplicial arrangements are $K(\pi, 1)$.
Likewise for complex fiber-type arrangements, 
cf.~\cite{falkrandell:fiber-type} and \cite{terao:modular}.
As restrictions of simplicial (resp.~ fiber-type)
arrangements are again simplicial 
(resp.~ fiber-type), 
the $K(\pi, 1)$-property of these kinds of arrangements
is inherited by their restrictions.
However, we emphasize that in general, a restriction of a 
$K(\pi, 1)$-arrangement need not be $K(\pi, 1)$ again, see 
\cite{amendmoellerroehrle:aspherical} for examples of this kind.

Along with the previously known instances of 
$K(\pi,1)$ restrictions of reflection arrangements,
it follows from Theorem \ref{thm:main} that 
only a small number of restrictions of rank $3$ and $4$ of arrangements associated to 
some non-real exceptional groups remain unresolved.
This provides strong evidence towards
the following.

\begin{Hope}
\label{conj:1}
Any arrangement induced from a reflection arrangement $\CA(W)$ is a $K(\pi,1)$.
\end{Hope}

This Hope reduces to the case of arrangements induced from reflection arrangements $\CA(W)$, for $W \subset \GL(V)$, such that $V$ is an irreducible representation of $W$. In what follows we, sometimes tacitly, only consider this case.

The Hope is true for $W$ the complexification of a real reflection group. Indeed, after reduction to the case where the intersection of all mirrors is reduced to \{0\}, such an $\CA(W)$ is the complexification of a simplicial arrangement (\cite{bourbaki:groupes} \Romannum{5} 3.9). The property of being the complexification of a simplicial arrangement is stable by induction, and one applies Deligne \cite{deligne}.

All reflection arrangements $\CA(W)$ are of $K(\pi,1)$ type. This theorem is due to 
Fadell and Neuwirth \cite{fadellneuwirth}, 
Brieskorn \cite{brieskorn:tresses},  
Nakamura \cite{nakamura}, Orlik and Solomon \cite{orliksolomon:discriminant} in special cases, and to Bessis \cite{bessis:kpione} in the general case. 

It follows from Theorem \ref{thm:main} that the arrangements induced from the reflection arrangements $\CA^0_{\ell}(r)$ and $\CA^{\ell}_{\ell}(r)$ are $K(\pi,1)$-arrangements.

Using those results and the trivial fact that for $\dim V \leq 2$ any arrangement is of $K(\pi,1)$ type, one gets 
our Hope in all but finitely many cases. 
More precisely, by Theorem \ref{thm:main}, our discussion above, and from the classification of the
irreducible complex reflection groups, our 
Hope reduces to $13$ instances 
when the underlying reflection group is 
of exceptional type.
Following 
\cite{orliksolomon:unitaryreflectiongroups},  
we label the $W$-orbit of $Y \in L(\CA(W))$ 
by the pair $(G_n, T)$, where $G_n$ is the relevant reflection group, 
in the Shephard-Todd numbering \cite{shephardtodd}, and $T$ is the 
type of the 
reflection subgroup of $G_n$ fixing  pointwise the  intersection of 
mirrors $Y$ we are considering.
In our cases, the type $T$ determines $Y$ up to $G_n$-conjugacy.
The 13 instances are
$(G_{29}, A_1)$, 
$(G_{31}, A_1)$,
$(G_{32}, C(3))$,
$(G_{33}, A_1)$,
$(G_{33}, A_1^2)$,
$(G_{33}, A_2)$,
$(G_{34}, A_1)$,
$(G_{34}, A_1^2)$,
$(G_{34}, A_2)$,
$(G_{34}, A_1^3)$,
$(G_{34}, A_1 A_2)$,
$(G_{34}, A_3)$, and 
$(G_{34}, G(3,3,3))$,
see \cite[\S 3, App.]{orliksolomon:unitaryreflectiongroups}.
The cases 
$(G_{32}, C(3))$ and $(G_{34}, G(3,3,3))$ 
can be handled as follows. The lattices of intersections of  
$(G_{32}, C(3))$ and $(G_{34}, G(3,3,3))$ are both isomorphic 
to the lattice of $\CA(G_{26})$, cf.~\cite[App.~D]{orlikterao:arrangements}.
Viewed projectively, the arrangement $\CA(G_{26})$ is the extended Hessian 
configuration of $21$ lines in $\BBP^{2}(\BBC)$, cf.~\cite[Ex.~6.30]{orlikterao:arrangements}.
It is classical that this configuration, as a set of $21$ lines, is determined by its combinatorics,
i.e.~by the isomorphism class of the corresponding lattice:
 the arrangements $(G_{32}, C(3))$ and $(G_{34}, G(3,3,3))$ are linearly isomorphic to the 
 reflection arrangement  $\CA(G_{26})$.
Therefore, since $\CA(G_{26})$ is a $K(\pi,1)$-arrangement, 
so are the restrictions $(G_{32}, C(3))$ and $(G_{34}, G(3,3,3))$.
Moreover, since the localization 
of a $K(\pi,1)$-arrangement is again  $K(\pi,1)$ (cf.~\cite[Lem.~1.1]{paris:deligne}),
 and since 
$G_{33}$ is a parabolic subgroup of  $G_{34}$ (\cite[Table 11]{orliksolomon:unitaryreflectiongroups}),
the three instances stemming from $G_{33}$ are localizations of 
the corresponding restrictions from  $G_{34}$. 
So verifying our Hope reduces 
to the remaining $8$ restrictions on the list above.

\section{Method of proof}
A reasonable topological space $X$, for instance a manifold or a CW complex, is a $K(\pi,1)$ if and only if it is connected (hence, by definition, not empty) and if for some (equivalently, any) base point $o\in X$, the homotopy groups $\pi_i(X,o)$ are trivial for $i \geqslant 2$. The long exact sequence of homotopy groups implies the 
\begin{lemma}
\label{lem:fibers}
Suppose that $X$ is connected and that $f\colon X \to Y$ is a fibration. Then, if $Y$ is a $K(\pi, 1)$ and if some connected component of some fiber $f^{-1}(y)$ is a $K(\pi, 1)$, so is $X$. 
\end{lemma}
We fix $k,\ \ell,\ r$ as in Theorem \ref{thm:main}. We simply write $\CA$ for the arrangement $\CA^k_{\ell}(r)$ in $\mathbb{C}^{\ell}$. The coordinates of $\mathbb{C}^{\ell}$ are denoted $y_1,\ldots, y_{\ell}$. 

Let us consider another copy of $\mathbb{C}^{\ell}$, with coordinates $x_1,\ldots,x_{\ell}$. Let $V$ be the quotient of this vector space $\BBC^{\ell}$ by its diagonal subspace $\BBC$. The action of the symmetric group $S_{\ell}$ on $\BBC^{\ell}$ passes to the quotient. So do the linear forms $x_i - x_j$. The \emph{arrangement} $\textrm{A}_{\ell-1}$ on $V$ is the set of the hyperplanes $x_i - x_j=0$ of $V$. It is the reflection arrangement $\CA(S_{\ell})$ defined by the action of $S_{\ell}$ on $V$. It is of $K(\pi,1)$ type, and the fundamental group of $X(\textrm{A}_{\ell-1})$ is the pure braid group on $\ell$ strands.

The $z_i\coloneqq x_i - x_{\ell}\ (1 \leq i \leq \ell - 1)$ form a system of coordinates on $V$. In this system of coordinates, the arrangement $\textrm{A}_{\ell-1}$ is the arrangement $\CA^{\ell-1}_{\ell-1}(1)$ consisting of the coordinate hyperplanes $z_i=0$ and of the hyperplanes $z_i=z_j$ for $i\ne j$.

Our \emph{deus ex machina} is the composite map 
\begin{equation}
\label{def:eq1}
f\colon \BBC^{\ell} \mathrm{(coordinates}\ y_i)\longrightarrow \BBC^{\ell}\text{(coordinates}\ x_i)\longrightarrow V,
\end{equation}
where the first map in \eqref{def:eq1}, or rather its graph, is given by
\begin{equation}
\label{def:eq2}
x_i = y_1\cdots y_k\ y^r_i.
\end{equation}
It is equivariant for the subgroup $S_k\times S_{\ell-k}$ of $S_{\ell}$, acting on $\BBC^{\ell}$ and on $V\colon$ the coordinates $y_1,\ldots,y_k$, as well as $y_{k+1},\ldots,y_{\ell}$, play symmetric roles.

The inverse image by $f$ of the union of the hyperplanes in $\textrm{A}_{\ell-1}$ is the union of the hyperplanes in $\CA$. Indeed, the inverse image of the hyperplane $x_i-x_j=0$ is the union of the coordinate hyperplanes $y_a=0$ for $1 \leq a \leq k$, and of the hyperplanes $y_i=\zeta y_j$ for $\zeta$ an $r\th$ root of unity.

In the coordinate system $(z_i)$ of $V$, the map \eqref{def:eq1} is given by
\begin{equation}
\label{def:eq3}
z_i = y_1\cdots y_k\ (y^r_i-y^r_{\ell}).
\end{equation}
It induces a map, still denoted by $f$
\begin{equation}
\label{def:eq4}
f\colon\ X(\CA)\longrightarrow X(\mathrm{A}_{\ell-1}).
\end{equation}
\begin{theorem}
\label{thm:2}
The map $f \colon\ X(\CA)\longrightarrow X(\mathrm{A}_{\ell-1})$ realizes $X(\CA)$ as a smooth fiber space over $X(\mathrm{A}_{\ell-1})$.
\end{theorem}
A consequence of Theorem \ref{thm:2} is that the fibers of $f\colon X(\CA) \longrightarrow X(\mathrm{A}_{\ell-1})$ are non empty smooth affine curves. The connected components of such a curve are again smooth and affine, hence are $K(\pi,1)$. Indeed, the only Riemann surface which is not a $K(\pi,1)$ is the sphere. As $X(\CA)$ is connected and $X(\mathrm{A}_{\ell-1})$ is a $K(\pi,1)$, Theorem \ref{thm:main} is a consequence of Theorem \ref{thm:2} and Lemma \ref{lem:fibers}.

The proof of Theorem \ref{thm:2} is given in the next section, where we use the following lemma.
\begin{lemma}
\label{lem:2.6}
Let $M$ and $B$ be $C^{\infty}$-manifolds, $N$ a closed submanifold of $M$ and $f \colon M \longrightarrow B$  a morphism. If $f$ is proper, submersive, and with a restriction to $N$ submersive, then, locally on $B, f \colon (M,N) \longrightarrow B$ is isomorphic to a projection $(M_0\times B, N_0 \times B)\longrightarrow B$. A fortiori, $f \colon M-N \longrightarrow B$ is a smooth fiber bundle.
\end{lemma}
For $N$ empty, the lemma first appeared without proof in  
\cite[Prop.~1]{ehresmann}. For the sake of completeness, we now explain the folklore proof of Lemma \ref{lem:2.6} in the 
case when $N$ is empty, and then explain how to extend it to the general case.

The question being local on $B$, we may assume that $B$ is of the form $\rbrack-1,1\lbrack^\ell$ (coordinates $t_i,\ldots, t_\ell)$, and we proceed by induction on $\ell$, the case $\ell=0$ being trivial. The vector field $\partial_{t_{\ell}}$ on $B$ can be lifted to a vector field $X$ on $M$. Indeed, such a lifting exists locally on $M$, and one uses a partition of unity to get a global lifting from local liftings. As $df(X)=\partial_{t_{\ell}}$, by integrating $X$ we obtain isomorphisms between the fibers of $f$ at ($t_1,\ldots,t_{\ell-1}, 0$) and ($t_1,\ldots,t_{\ell-1}, t_{\ell}$). These isomorphisms identify $M\longrightarrow B$ with the pull-back by $\rbrack-1, 1\lbrack^{\ell}\longrightarrow \rbrack-1, 1\lbrack^{\ell-1}$ of the restriction of $M\longrightarrow B$ to $\rbrack-1, 1\lbrack^{\ell-1}\times\lbrace 0\rbrace \subset B$. One concludes using the induction hypothesis.

The proof of Lemma \ref{lem:2.6} is identical : one just needs to choose the lifting $X$ of $\partial_{t_{\ell}}$ to be tangent to $N$.
\section{Proof of Theorem \ref{thm:2}}
The fiber $F_z$ at $z \in X(\mathrm{A}_{\ell-1})$ of $f\colon\ X(\CA) \longrightarrow X (\mathrm{A}_{\ell-1})$ is given, in $\BBC^{\ell}$, by the equations
\begin{equation}
    \tag{1}
    \label{eq:5}
y_1\cdots y_k (y^r_i-y^r_{\ell})=z_i \qquad (i=1,\ldots,\ell-1).
\end{equation}
Any of these equations implies that $y_1,\ldots, y_k \ne 0$. Their system is equivalent to the first equation
\begin{equation}
    \tag{2}
    \label{eq:6}
y_1 \cdots y_k (y^r_1 - y^r_{\ell}) = z_1 ,
\end{equation}
supplemented by the equations
\begin{equation}
    \tag{3}
    \label{eq:7}
\frac{1}{z_1} (y^r_1 - y^r_{\ell}) = \frac{1}{z_i} (y^r_i - y^r_{\ell})\qquad(2\leq i \leq \ell - 1)
\end{equation}
which are homogeneous of degree $r$ in the $y_i$. 

Let us compactify $\BBC^{\ell}$ into $\BBP^{\ell}(\BBC)$. In $\BBP^{\ell}(\BBC)$, we will use the homogeneous coordinates $y_0, y_1, \ldots, y_{\ell},\ y_0 = 0$ being the equation of the hyperplane at infinity added to $\BBC^{\ell}$. To compactify the fiber $F_z$ it suffices to take the projective variety $\overline{F}_z$ defined by the homogeneous equations (\ref{eq:7}), and by (\ref{eq:6}) made homogeneous, that is
\begin{equation}
    \tag{2\textprime}
    \label{eq:6.1}
y_1 \cdots y_k (y^r_1 - y^r_{\ell}) = z_1\ y^{k+r}_0,
\end{equation}
an equation homogeneous of degree $k+r$ in the $y_i$.

It will be convenient to define $z_{\ell}\coloneqq 0$. With this notation, (\ref{eq:7}) tells that the $y^r_i - y^r_{\ell}$ are proportional to the $z_i - z_{\ell}$  and it follows that all $y^r_i-y^r_j$ are proportional to the $z_i-z_j$ : for some $u,\ y^r_i - y^r_j = u (z_i - z_j)$. 

To compute the intersection of this compactification $\overline{F}_z$ with the hyperplane at infinity $H_{\infty}$, it suffices to put $y_0$ equal to 0, and to view $y_1,\ldots, y_{\ell}$ as projective coordinates for the hyperplane at infinity. We obtain $kr^{\ell-2}+r^{\ell-1}$ distinct points, as follows. One of the factors at the left side of (\ref{eq:6.1}) must vanish. If $y_i=0$  
$ (1 \leq i \leq k)$, the $y^r_j=y^r_j-y^r_i$ are proportional to the $z_j - z_i$, and we get the $r^{\ell-2}$ points with coordinates $$(y_i=0, y_j= (z_j-z_i)^{\frac{1}{r}})$$ (to be taken up to multiplying by a common $r\th$ root of $1$).\\
If $y^r_1-y^r_{\ell}=0$, all $y^r_i - y^r_j$ must vanish. We get the $r^{\ell-1}$ points "all $y_i$ are an $r\th$ root of 1", again taken up to multiplication by a common $r\th$ root of 1.
\begin{lemma}
\label{lem:3.2}
The compactification $\overline{F}_z$ of the fiber $F_z$ of $f$ at $z$ defined by the $(\ell-1)$ equations \emph{(\ref{eq:6.1})} and \emph{(\ref{eq:7})} is a complete intersection curve, smooth at infinity, and meeting transversally the hyperplane at infinity $H_{\infty}$.
\end{lemma}
\begin{proof}
If $\overline{F}_z$ had an irreducible component of dimension $>1$, the intersection of this component with $H_{\infty}$ would be of dimension $>0$, contradicting the finiteness of $\overline{F}_z \cap H_{\infty}$. It follows that $\overline{F}_z$, being defined by $(\ell-1)$ equations, is a complete intersection curve. By Bezout, the number of points in $\overline{F}_z \cap H_{\infty}$, each counted with its intersection multiplicity, is $(k+r)r^{\ell-2}$. It follows that each intersection multiplicity is one.
As $\overline{F}_z$ and $H_{\infty}$ are local complete intersections, this implies that 
 $\overline{F}_z$ is smooth at each point of $\overline{F}_z \cap H_{\infty}$, and that the intersection is transversal. 
\end{proof}

It follows from Lemma~\ref{lem:3.2} that $\overline{F}_z$ is simply the closure of $F_z$ in $\mathbb{P}^{\ell}(\BBC)$, and that the curve $\overline{F}_z$ is generically reduced, that is generically smooth, as a non reduced component would intersect $H_{\infty}$.

The same argument shows that
\begin{lemma}
\label{lem:3.3}
For $k+1\leqslant i \leqslant \ell$, the curve $\overline{F}_z$ is smooth at each of its intersection points with the hyperplane $y_i=0$.
\end{lemma}

\begin{proof}
It suffices to show that the number of intersection points is $(k+r)r^{\ell -2}$. As the hyperplanes $y_i=0\ (k + 1 \leqslant i \leqslant \ell)$ play symmetric roles, it suffices to consider the case of the hyperplane $y_{\ell}=0$. The equations (\ref{eq:7}) tell us that $(y^r_1, \ldots,y^r_{\ell-1})$ is proportional to $(z_1,\ldots,z_{\ell-1})$, while by (\ref{eq:6.1}) they cannot be all zero, as otherwise $y_0$ would be zero too. If we fix the indeterminacy "multiplication by a common constant" by requiring $y_{\ell-1}$ to be a specified root of $z_{\ell-1}$, (\ref{eq:7}) tells that each $y_i\ (1\leqslant i \leqslant \ell-2)$ is an $r\th$ root of $z_i$. This gives $r^{\ell-2}$ possibilities, while (\ref{eq:6.1}) leaves $(k+r)$ possibilities for $y_0$.
\end{proof}

\begin{lemma}
\label{lem:3.4}
The curve $\overline{F}_z$ is smooth.
\end{lemma}

\begin{proof}
By Lemma \ref{lem:3.2}, it suffices to show that $F_z$ is smooth. By (\ref{eq:6}), $F_z$ does not intersect the hyperplanes $y_i=0$ for $1\leq i \leq k$. By Lemma \ref{lem:3.3}, it hence suffices to check that in the open set where none of the $y_i$ vanishes, $F_z$ is smooth. Locally on $(\BBC^*)^\ell$, we can take as local coordinates the $Y_i=y^r_i$. In these local coordinates, (\ref{eq:7}) tells us that $F_z$ is on the surface $Y_i=az_i+b$ (coordinates $a,b$). The last equation (\ref{eq:6}) becomes $$\prod_{i=1}^k (az_i+b)^{\frac{1}{r}} \cdot az_1=z_1$$ for some branches of the $r\th$ roots. It follows that $a\ne 0$, and that $F_z$ is contained in the curve of the plane $(a,b)$ $$a^r\prod_{i=1}^k (az_i+b)=1.$$ 

One concludes by invoking the well known 
\begin{lemma}
\label{lem:3.5}
If $F(a_1,\ldots,a_n)$ is a homogeneous polynomial of degree $d\geqslant 1$, the hypersurface $F(a_1,\ldots,a_n)=1$ is non singular.
\end{lemma}

By homogeneity, the hypersurfaces $F=c$ are, for $c\ne 0$, all isomorphic. By Sard's theorem, almost all are non singular. One could rather use the Jacobian criterion : at a point where $F=1$, Euler's identity $$\Sigma\ a_i \partial_i F\ = dF\ =d$$ shows that not all $\partial_iF$ can vanish. 

If we now let $z$ vary in $X(\mathrm{A}_{\ell-1})$, we obtain a family of smooth complete intersection curves in the projective space $\mathbb{P}^{\ell}(\BBC)$ completing $\BBC^{\ell}$, transversal to the hyperplane at infinity. The total space is contained in $\mathbb{P}^{\ell}(\BBC) \times X(\mathrm{A}_{\ell-1})$, and one applies Lemma \ref{lem:2.6} to it.
\end{proof}

\section{Complements}
\subsection{} The projective completion $\overline{F}_z$ of the fiber $F_z$ of the fiber bundle $f\ \colon\ X(\CA) \longrightarrow X(\mathrm{A}_{\ell-1})$ is a smooth complete intersection in $\mathbb{P}^{\ell}(\BBC)$, of multidegree $(k+r,r,\ldots,r)$. 
As $\overline{F}_z$ is defined by $(\ell-2)$ homogeneous equations, its complement $U$ is the union of $(\ell-2)$ smooth open affine varieties of dimension $\ell$, so that $H^1_c(U) =0$ and 
$H^0(\mathbb{P}^{\ell}(\BBC)) \stackrel{\sim}{\to} H^0(\overline{F}_z)$: so $\overline{F}_z$ is connected. To see this, one could rather make an iterated application to the Lefschetz hyperplane theorem, again for $H^0$.  

The canonical line bundle of a complete intersection $Y$ of degrees $(d_1,\ldots,d_{\ell-1})$ in $\mathbb{P}^{\ell}(\BBC)$ 
is isomorphic to the restriction to $Y$ of $\CO(\Sigma\ d_i - \ell - 1)$. In our case, it follows that the degree $2g-2$ of the canonical line bundle of the curve $\overline{F}_z$ is given by 
\[
2g-2 = (k+(r-1)(\ell - 1) - 2)(k+r)r^{\ell-2}.
\] 
The fiber $F_z$ is the complement in $\overline{F}_z$ of $(k+r)r^{\ell-2}$ points. Its fundamental group is hence a free group with $N$ generators, where 
\begin{align*}
N &= 2g\ +  \text{number of removed points}\quad -1\\
&= (k+(r-1)(\ell-1)-1)\ (k+r)r^{\ell-2}\quad + 1.
\end{align*} 
\subsection{}
Each curve $\overline{F}_z$ contains at infinity the point $(0,1,\ldots,1)$. If $M\longrightarrow X(\mathrm{A}_{\ell-1})$ is the total space of the family of the $\overline{F}_z$ (contained in $\mathbb{P}^{\ell}(\BBC)\times X(\mathrm{A}_{\ell-1})$), this common point gives us a section $s$ of the fiber bundle $M\longrightarrow X(\mathrm{A}_{\ell-1})$. The vertical tangent bundle, restricted to this section, is a trivial line bundle, because any line bundle on $X(\mathrm{A}_{\ell-1})$ is trivial. Let $v$ be, along $s$, a nowhere vanishing section of the vertical tangent bundle. Pushing $s$ in the direction of $v$, one obtains a $C^{\infty}$ section of $X(\CA)\longrightarrow X(\mathrm{A}_{\ell-1})$. The fundamental group of $X(\CA)$ is hence a semi-direct product of the fundamental group of the basis by the fundamental group of the fiber : a semi-direct product of the braid group on $\ell$ strands by the free group on $N$ generators.
\bigskip


\bigskip {\bf Acknowledgments}: 
The first and last authors acknowledge 
support from the DFG-priority program 
SPP1489 ``Algorithmic and Experimental Methods in
Algebra, Geometry, and Number Theory''.

\bibliographystyle{amsalpha}

\newcommand{\etalchar}[1]{$^{#1}$}
\providecommand{\bysame}{\leavevmode\hbox to3em{\hrulefill}\thinspace}
\providecommand{\MR}{\relax\ifhmode\unskip\space\fi MR }
\providecommand{\MRhref}[2]{%
  \href{http://www.ams.org/mathscinet-getitem?mr=#1}{#2} }
\providecommand{\href}[2]{#2}

\end{document}